\documentclass[11pt]{amsart}
\usepackage[T1]{fontenc}
\usepackage{lmodern}
\usepackage{amsmath,amssymb,amsthm,mathtools}
\usepackage{enumitem}
\usepackage{tikz}
\usepackage[hidelinks]{hyperref}
\usepackage{microtype}
\theoremstyle{plain}
\newtheorem{theorem}{Theorem}[section]
\newtheorem{lemma}[theorem]{Lemma}
\newtheorem{proposition}[theorem]{Proposition}
\newtheorem{corollary}[theorem]{Corollary}
\theoremstyle{remark}
\newtheorem{remark}[theorem]{Remark}
\newtheorem{example}[theorem]{Example}

\newcommand{\Spec}{\operatorname{Spec}}
\newcommand{\height}{\operatorname{ht}}

\newcommand{\fm}{\mathfrak m}
\newcommand{\fp}{\mathfrak p}
\newcommand{\fq}{\mathfrak q}
\newcommand{\fa}{\mathfrak a}
\newcommand{\fn}{\mathfrak n}
\newcommand{\fM}{\mathfrak M}
\newcommand{\Q}{\mathbb Q}

\hypersetup{pdftitle={Which posets are spectra of Noetherian rings? The case of height at most one},
  pdfauthor={Xin Li},
  pdfkeywords={prime spectrum, finite poset, bipartite graph, Noetherian ring}}

\title[Posets that are Noetherian spectra]{Which posets are spectra of Noetherian rings? The case of height at most one}
\author{Xin Li}
\address{Shenzhen University of Advanced Technology, Shenzhen, China}
\email{suat25000420@stu.suat-sz.edu.cn}
\date{}

\begin{document}

\begin{abstract}
A finite poset is the prime spectrum of a Noetherian ring exactly when it has no chain of three elements, that
is, it is the poset of a finite bipartite graph. This answers Conjecture~5.3 of Alkass. The substance is a
construction: minimal elements become curves $y=g_i(x)$, maximal elements points, edges the
incidences, arranged by interpolation. The coordinate ring of the union of the curves, localized at the chosen
points, is essentially of finite type over $\Q$ and of dimension at most one. Dropping finiteness costs two
conditions, both isolated by Colbert: a poset with no chain of three elements is a Noetherian spectrum exactly
when it has finitely many
minimal elements and any two distinct ones have finitely many common upper bounds. A variant keeping the shape of
the ideals of Alkass but changing their exponents shows why that construction produces only
nested neighborhoods.
Constraining the ring destroys most of this freedom: normal rings leave finite disjoint unions of stars, domains a
single star, local rings a point or a finite co-star. In height two the question is open in general, but the conditions conjectured by Wiegand and Wiegand are not sufficient.
\end{abstract}

\maketitle

\section{Introduction}\label{sec:intro}

Kaplansky asked which partially ordered sets are prime spectra of Noetherian rings, as \cite[\S1]{A} records.
Without the Noetherian hypothesis no constraint appears: the spectral spaces are exactly the spectra of rings
\cite{H}, and every finite poset, with the Alexandrov topology, is one of them. With it the constraint is
severe. A Noetherian ring of dimension at least two has infinitely many primes \cite[Lemma~10.61.2]{Stacks}, so
a Noetherian ring with finitely many primes has dimension at most one. Its spectrum therefore has no chain of
three elements: every prime is minimal or maximal, and the comparabilities form the edges of a bipartite graph
whose parts are the minimal and the maximal elements.

Alkass proved that complete bipartite graphs, paths and even cycles, among other families, arise in this
way, and conjectured that every finite bipartite graph is order-isomorphic to the spectrum of a ring essentially of finite type
over a field \cite[Conjecture~5.3]{A}. The conjecture holds, and the condition is exact: a finite poset with a
chain of three elements is not the spectrum of a Noetherian ring, so the hypothesis cannot be weakened.

\begin{theorem}\label{thm:main}
For a finite poset $P$ the following are equivalent.
\begin{enumerate}[label=(\roman*),itemsep=1pt,topsep=2pt]
\item $P$ is isomorphic to $\Spec R$ for some Noetherian ring $R$.
\item $P$ is isomorphic to $\Spec R$ for some ring $R$ essentially of finite type over a field.
\item $P$ has no chain of three elements.
\item $P$ is the poset of a finite bipartite graph (Remark~\ref{rem:dictionary}).
\end{enumerate}
Moreover, in (ii) the field may be taken to be $\Q$, and then $\Spec R$ has dimension at most one, with
dimension one exactly when $P$ has an edge; for $P=\emptyset$ one takes $R=0$.
\end{theorem}

\begin{remark}[the dictionary]\label{rem:dictionary}
Let $B=(X,Y;E)$ be a finite bipartite graph with parts $X$ and $Y$. Its \emph{poset} has elements
$X\sqcup Y$, with $x<y$ for each edge $xy\in E$ and no other comparabilities. Conversely, every finite poset
$P$ with no chain of three elements is the disjoint union of the poset of the bipartite graph with parts
$X=\min P\setminus\max P$ and $Y=\max P\setminus\min P$, whose edges are the comparabilities, and the
antichain $A_0=\min P\cap\max P$ of isolated elements. This is a bijection
between such posets and pairs consisting of a bipartite graph without isolated vertices and a finite antichain,
so (iii) and (iv) of Theorem~\ref{thm:main} say the same thing.

For cardinals $m,n$ we write $K_{m,n}$ for the poset with elements $x_i$ $(i<m)$ and $y_j$ $(j<n)$, ordered
by $x_i<y_j$ for all $i,j$ and nothing else; it is the poset of the complete bipartite graph with parts of
sizes $m$ and $n$, and $\infty$ in a subscript means countably infinite. We call $K_{1,t}$ a \emph{star} and
$K_{s,1}$ a \emph{co-star}.
\end{remark}

The implication (ii)$\Rightarrow$(i) is immediate and (iii)$\Leftrightarrow$(iv) is Remark~\ref{rem:dictionary},
so the content is (iii)$\Rightarrow$(ii), together with the reverse implication (i)$\Rightarrow$(iii), the dimension count of
the first paragraph. This paper gives two constructions for (iii)$\Rightarrow$(ii), one geometric in
Section~\ref{sec:geometric} and one arithmetic in Section~\ref{sec:arithmetic}. The geometric one appears to be
new; both constructions give rings essentially of finite type over $\Q$. For the finite theorem, even the bare
existence statement is a theorem of Colbert \cite{C18} in dimension at most one, as
Section~\ref{sec:relation} explains.

The idea for (iii)$\Rightarrow$(ii) is to build the graph from geometry. Put the maximal elements at
points $P_j=(a_j,b_j)$ of the plane, with distinct $a_j$, and the minimal elements on curves
$C_i=\{y=g_i(x)\}$, and let the edges be the incidences. A curve of this form is irreducible without further
hypotheses, since $y-g_i(x)$ is monic of degree one in $y$. Interpolation lets us choose the value of $g_i$ at every $a_j$, so the incidences can be
prescribed at will. Take the coordinate ring of the union of the
curves, localized at the complement of the union of the maximal ideals of the chosen points: its only primes are the curves and the points. Example~
\ref{ex:two} works this out on four elements.

Dropping finiteness costs two conditions, and no more, the first two in Colbert's definition of a K-poset \cite{C18}: a poset of height at most one, that is, with no chain of
three elements, is a Noetherian spectrum exactly when it has finitely many minimal elements and any two
distinct ones have only finitely many common upper bounds (Theorem~\ref{thm:infinite}). Neither condition
follows from the other: $K_{\infty,1}$ has infinitely many minimal elements below a single maximal one, and
$K_{2,\infty}$, in which countably many maximal elements lie above the same two minimal ones, fails the second.

The ring constrains the poset as well. A poset with no chain of three elements is the spectrum of a Noetherian
normal ring if and only if it is a disjoint union of finitely many stars $K_{1,t}$ with $t$ an arbitrary
cardinal (Theorem~\ref{thm:normal}); it is the spectrum of a Noetherian domain (equivalently, of a Noetherian
unique factorization domain) if and only if it has a single minimal element (Theorem~\ref{thm:ufd}); and it is
the spectrum of a Noetherian local ring if and only if it is a single point or a finite co-star $K_{s,1}$
(Theorem~\ref{thm:local}). So the freedom of Theorem~\ref{thm:main} disappears in all three directions.

Sections~\ref{sec:prelim} to~\ref{sec:converse} prove Theorem~\ref{thm:main}. Sections~\ref{sec:infinite} and~\ref{sec:properties}
can be read independently of one another and of Section~\ref{sec:converse}:
they characterize the posets of height at most one and ask which of them stay when the ring is required to be
normal, a domain, a unique factorization domain or a local ring. Section~\ref{sec:heighttwo} takes up height two: it gives a criterion in the countable case, and it
refutes Conjecture~1.4 of \cite{WW}, which had proposed different conditions for the same spectra.
Section~\ref{sec:relation} compares the result with the literature.

\section{Preliminaries}\label{sec:prelim}

Rings are commutative with $1$, the zero ring is allowed and counts as reduced and normal, and $\Spec R$ is
the set of prime ideals of $R$ ordered by inclusion.\footnote{The term \emph{Noetherian spectrum} also names a
topological property, that every open set of a space be quasi-compact \cite{OP}. In this paper a Noetherian
spectrum is the spectrum of a Noetherian ring.} A ring is \emph{semilocal} when it has finitely many maximal
ideals, and \emph{essentially of finite type} over a field $k$ when it is a localization of a finitely generated
$k$-algebra. In this paper a poset has height at most one when it has no chain of three elements, and we use the two
descriptions interchangeably; the \emph{height} of an element is the supremum of the lengths of the chains of
elements strictly below it, so the minimal elements are those of height zero.

We use three standard facts. First, prime avoidance: a prime ideal contained in a finite union of prime ideals
is contained in one of them \cite[Proposition~1.11]{AM}. Second, Krull's principal ideal theorem: in a Noetherian
ring a prime minimal over a principal ideal has height at most one, and height zero is the same as minimal
\cite[Corollary~11.17]{AM}, \cite[Theorem~13.5]{Ma}. Third, used in Sections~\ref{sec:geometric}
and~\ref{sec:arithmetic}: for a field $k$ and a
nonconstant $f\in k[x,y]$, the ring $k[x,y]/(f)$ has dimension one, and if $f$ is prime it is a
one-dimensional domain, so each of its nonzero primes is maximal. When $f=y-g(x)$ is linear in $y$,
evaluation of $y$ at $g(x)$ gives $k[x,y]/(f)\cong k[x]$.

The next lemma is \cite[Lemma~3.1]{A}, in the form we need.

\begin{lemma}\label{lem:localization}
Let $A$ be a ring, $t\ge1$, $n\ge1$, and let $p_1,\dots,p_n$ and $\fm_1,\dots,\fm_t$ be prime ideals of $A$
such that each $\fm_j$ contains some $p_i$. Put $\fa=p_1\cdots p_n$ and let
$S=(A/\fa)\setminus\bigcup_j\overline{\fm_j}$, where the bar denotes the image in $A/\fa$. Then
$\Spec\bigl(S^{-1}(A/\fa)\bigr)$ is order-isomorphic to the set of primes $q$ of $A$ with
$p_i\subseteq q\subseteq\fm_j$ for some $i$ and $j$, ordered by inclusion.
\end{lemma}

\begin{proof}
A prime $q$ of $A$ contains $\fa$ if and only if it contains some $p_i$: if $\fa\subseteq q$ but
$p_i\not\subseteq q$ for every $i$, then choosing $x_i\in p_i\setminus q$ gives
$x_1\cdots x_n\in\fa\setminus q$. So $\Spec(A/\fa)$ is the set of primes containing some $p_i$, ordered by
inclusion.

Each $\overline{\fm_j}$ is a proper prime of $A/\fa$. It is prime because
$(A/\fa)/(\fm_j/\fa)\cong A/\fm_j$ is a domain, and proper because $\fa\subseteq p_i\subseteq\fm_j$ for a
suitable $i$. Hence $1\in S$, and $S$ is multiplicative: if $xy$ lies in the union, it lies in some
$\overline{\fm_j}$, and then $x$ or $y$ does.

A prime $\bar q$ of $A/\fa$ is disjoint from $S$ exactly when it is contained in the union of the
$\overline{\fm_j}$, and by prime avoidance exactly when it is contained in one of them. Localizing at $S$
therefore keeps precisely the primes of $A/\fa$ that lie in some $\overline{\fm_j}$, and the
correspondence preserves inclusions \cite[Proposition~3.11]{AM}.
\end{proof}

\begin{remark}\label{rem:hypothesis}
The hypothesis that each $\fm_j$ contains some $p_i$ gives $\fa\subseteq\fm_j$, so that $\overline{\fm_j}$ is
a proper prime and a prime $q\supseteq\fa$ satisfies $\bar q\subseteq\overline{\fm_j}$ exactly when
$q\subseteq\fm_j$; this is what makes the two sides of the lemma match. Without it they can differ: for
$A=k[x,y]$, $\fa=(y)$ and $\fm=(x)$ the image $\overline{\fm}=(x)$ is a proper prime of $A/\fa\cong k[x]$, and
the localization at the complement of $(x)$ has the two primes $(0)$ and $(x)$, whereas no prime $q$ satisfies
$(y)\subseteq q\subseteq(x)$.
\end{remark}

\section{The geometric construction}\label{sec:geometric}

Throughout this section $P$ is a finite poset with no chain of three elements. By
Remark~\ref{rem:dictionary}, $P$ is the disjoint union of the poset $(X\sqcup Y;E)$ of a bipartite graph
without isolated vertices and an antichain $A_0$, where $E$ is the set of pairs $x<y$.

An element of $A_0$ is incomparable with every other element, so it contributes a point of its own:
$\Spec(\Q^m)$ is a discrete set of $m$ points, and spectra of products of nonzero rings are disjoint unions
\cite[Lemma~3.9]{A}, \cite[Lemma~10.21.2]{Stacks}. If $(X\sqcup Y;E)$ is empty, $\Q^{|A_0|}$ realizes the
antichain $P$, of dimension $0$ unless $P=\emptyset$. Otherwise let $R_0$ realize $(X\sqcup Y;E)$ and take
$R=R_0$, or $R_0\times\Q^{|A_0|}$ when $A_0\ne\emptyset$; then $\dim R=\dim R_0$ and $R$ is
essentially of finite type over $\Q$.

It remains to realize a bipartite graph without isolated vertices, so from now on $X$ and $Y$ are nonempty
and no vertex is isolated. Write $X=\{x_1,\dots,x_n\}$, $Y=\{y_1,\dots,y_t\}$ and
$N(i)=\{j:(i,j)\in E\}$, the \emph{neighborhood} of $x_i$; by hypothesis $n,t\ge1$, and $N(i)$ and
$M(j)=\{i:j\in N(i)\}$ are nonempty for all $i$ and $j$.

\subsection*{The data}

Choose distinct rationals $a_1,\dots,a_t$ and distinct rationals $b_1,\dots,b_t$, and put
$P_j=(a_j,b_j)$ and $\fm_j=(x-a_j,y-b_j)$. Then $\fm_j$ is the kernel of evaluation at $P_j$, hence maximal,
because evaluation is onto $\Q$.

Let $h(x)=\prod_j(x-a_j)$. For each $i$, let $g_i^0$ be the interpolation polynomial of degree $<t$ with
$g_i^0(a_j)=b_j$ for $j\in N(i)$ and $g_i^0(a_j)=b_j+1$ otherwise, and set
\[
g_i=g_i^0+c_i h,\qquad f_i=y-g_i(x),\qquad p_i=(f_i),
\]
where $c_1,\dots,c_n$ are distinct nonzero rationals; over $\Q$ one may take $c_i=i$. These constants play
no role in the incidences and serve only to keep the curves distinct when two minimal elements have the
same neighbors.

\begin{lemma}\label{lem:data}
With this notation:
\begin{enumerate}[label=(\roman*),itemsep=1pt,topsep=2pt]
\item $g_i(a_j)=b_j$ if and only if $j\in N(i)$;
\item the $g_i$ are pairwise distinct, even when $N(i)=N(i')$;
\item $f_i$ is irreducible, so $p_i$ is prime;
\item the $p_i$ are distinct and pairwise incomparable, and $p_i\subseteq\fm_j$ if and only if
$(i,j)\in E$.
\end{enumerate}
\end{lemma}

\begin{proof}
(i) Since $h$ vanishes at every $a_j$, $g_i(a_j)=g_i^0(a_j)$, which equals $b_j$ exactly for $j\in N(i)$.

(ii) The polynomial $g_i^0$ has degree $<t$ and $h$ is monic of degree $t$, so the coefficient of $x^t$ in
$g_i$ is $c_i$.

(iii) The polynomial $f_i$ is primitive of $y$-degree one over the UFD $\Q[x]$. If $f_i=uv$ with neither
factor a unit, the $y$-degrees add to one, so one factor lies in $\Q[x]$ and divides both coefficients $1$
and $g_i$; hence it is a constant. So $f_i$ is irreducible, and since $\Q[x,y]$ is a UFD, $p_i$ is prime.

(iv) If $p_i\subseteq p_{i'}$, then $f_i=f_{i'}w$. Comparing $y$-degrees gives $w\in\Q[x]$, and comparing
the coefficients of $y$ gives $w=1$. So $f_i=f_{i'}$ and $g_i=g_{i'}$, which by (ii) forces $i=i'$.
Finally $p_i\subseteq\fm_j$ if and only if $f_i$ vanishes at $P_j$, that is, if and only if
$g_i(a_j)=b_j$, which by (i) means $j\in N(i)$.
\end{proof}

\subsection*{The ring}

Set $\fa=\prod_i p_i=(f_1\cdots f_n)$ and let $R=S^{-1}(\Q[x,y]/\fa)$, where $S$ is the complement of the
union of the images of the $\fm_j$.

\begin{lemma}\label{lem:spectrum}
The primes of $R$ are the images of $p_1,\dots,p_n,\fm_1,\dots,\fm_t$, with
$p_i\subseteq\fm_j$ exactly when $(i,j)\in E$ and with no other inclusion among them. In particular
$p_i\prec\fm_j$ is a cover for every edge, and $\Spec R\cong(X\sqcup Y;E)$.
\end{lemma}

\begin{proof}
By Lemma~\ref{lem:localization}, $\Spec R$ is the set of primes $q$ of $\Q[x,y]$ with
$p_i\subseteq q\subseteq\fm_j$ for some $i,j$. Every $p_i$ and every $\fm_j$ is such a prime, because every
vertex has a neighbor.

Let $q$ be such a prime, say $p_i\subseteq q\subseteq\fm_j$. If $q=p_i$ we are done; if $q=p_{i'}$ with
$i'\ne i$, then $p_i\subseteq p_{i'}$, which Lemma~\ref{lem:data}(iv) forbids. Otherwise $q/p_i$ is a
nonzero prime of $\Q[x,y]/(f_i)\cong\Q[x]$, hence maximal, and it is contained in the prime
$\fm_j/p_i$ of that ring, hence equal to it. So $q=\fm_j$.

This shows that the primes of $R$ are exactly the listed ones, and that the interval between $p_i$ and
$\fm_j$ has two elements when $(i,j)\in E$. The remaining claims are Lemma~\ref{lem:data}(iv), the
distinctness and maximality of the $\fm_j$, and the fact that a maximal ideal cannot be contained in the
non-maximal ideal $p_i$.
\end{proof}

The ring $R$ is a localization of $\Q[x,y]/(f_1\cdots f_n)$, so $\dim R\le1$, and $\dim R=1$ as soon as
$P$ has an edge, since an edge gives a chain $p_i\subsetneq\fm_j$. For a nonempty poset without edges the
reduction above gives $\Q^{|P|}$, of dimension $0$.

\begin{example}\label{ex:two}
Take $a_1=1,b_1=3,a_2=2,b_2=6$, so $\fm_1=(x-1,y-3)$, $\fm_2=(x-2,y-6)$ and $h=(x-1)(x-2)$.

If both minimal elements are adjacent to both points, then $g_1^0=g_2^0=3x$, and the two curves are
$g_1=x^2+2$ and $g_2=2x^2-3x+4$. Both pass through both points, so $\Spec R\cong K_{2,2}$; the constants
$c_i$ are what keep the curves apart here, the two neighborhoods being equal.

If instead $x_1$ is adjacent to both points and $x_2$ only to $P_2$, then $g_1^0=3x$ as before and
$g_2^0=2x+2$, so $g_2=2x^2-4x+6$. Now $g_2(1)=4\ne3$ and $g_2(2)=6$, so the second curve passes through
$P_2$ alone, and $\Spec R$ is the path $\fm_1>p_1<\fm_2>p_2$ (Figure~\ref{fig:path}).
\end{example}

\begin{figure}[t]
\centering
\begin{tikzpicture}[x=1.5cm,y=0.48cm]
  \draw[thick,domain=0.6:2.45,smooth,variable=\x] plot ({\x},{\x*\x+2});
  \draw[thick,domain=0.6:2.45,smooth,variable=\x] plot ({\x},{2*\x*\x-4*\x+6});
  \node[left] at (0.6,2.4) {$C_1$};
  \node[left] at (0.6,4.3) {$C_2$};
  \fill (1,3) circle (2.8pt) node[below right] {$P_1$};
  \fill (2,6) circle (2.8pt) node[below right] {$P_2$};
\end{tikzpicture}
\hspace{1.6cm}
\begin{tikzpicture}[scale=0.62]
  \node (p1) at (0,0) {$p_1$};
  \node (p2) at (2.0,0) {$p_2$};
  \node (m1) at (0,1.8) {$\fm_1$};
  \node (m2) at (2.0,1.8) {$\fm_2$};
  \draw (p1) -- (m1); \draw (p1) -- (m2); \draw (p2) -- (m2);
\end{tikzpicture}
\caption{The second case of Example~\ref{ex:two}, drawn from $g_1=x^2+2$ and $g_2=2x^2-4x+6$: $C_1$ passes
through both points, $C_2$ only through $P_2$. In $\Spec R$ the element $p_1$ therefore lies below both maximal
elements, and $p_2$ only below $\fm_2$.}
\label{fig:path}
\end{figure}
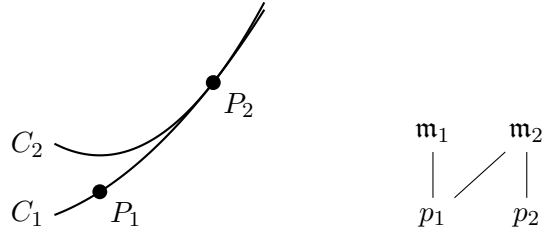

\subsection*{The geometric meaning}

The construction is the localization of $\Q[x,y]/(f_1\cdots f_n)$ that keeps the points $P_j$, that is, the
semilocal ring of the union of the curves at those points; and nothing in it needs the curves to be graphs of
polynomials.

\begin{theorem}\label{thm:geometric}
Let $C_1,\dots,C_n$ be distinct irreducible plane curves over $\Q$ with pairwise non-associate defining
polynomials $f_i$, let
$P_1,\dots,P_t$ be distinct rational points, and assume that every curve contains at least one of the
points and that every point lies on at least one curve. Put $R=S^{-1}\bigl(\Q[x,y]/(f_1\cdots f_n)\bigr)$,
where $S$ is the complement of the union of the images of the maximal ideals of the points $P_j$. Then
$\Spec R$ is the incidence poset of the configuration.
\end{theorem}

\begin{proof}
The empty configuration gives $R=0$, so suppose $n,t\ge1$. The assertion is the argument of Lemma~\ref{lem:spectrum}: for pairwise non-associate irreducible $f_i$ the minimal
primes of $\Q[x,y]/(f_1\cdots f_n)$ are the $(f_i)$, each $\Q[x,y]/(f_i)$ is a one-dimensional domain, and
$\fm_{P_j}/(f_i)$ is maximal whenever $P_j\in C_i$. Both hypotheses are needed: they keep the localization
nonzero and prevent a vertex from disappearing.
\end{proof}

\subsection*{What the construction gives}

The construction loses nothing: every finite bipartite graph without isolated vertices is the poset of such a
configuration, with all curves graphs of polynomials, since a curve $y=g(x)$ meets the line $x=a_j$ at
$(a_j,g(a_j))$ and interpolation makes $g(a_j)$ any prescribed value; adjoining isolated points and the empty
poset are covered by $\Q^{|A_0|}$ and by the zero ring in the first paragraph of the section.

\begin{corollary}\label{cor:domain}
Let $P$ be a finite poset with no chain of three elements. Then $P$ is the spectrum of a Noetherian domain if
and only if it has a single minimal element, and in that case $P=K_{1,t}$, with the domain $\Q$ if $t=0$ and a
one-dimensional domain essentially of finite type over $\Q$ if $t\ge1$. Every nonempty finite poset with no
chain of three elements is the spectrum of a reduced semilocal ring of dimension at most one, essentially of
finite type over $\Q$, of dimension one exactly when $P$ has an edge.
\end{corollary}

\begin{proof}
A domain has exactly one minimal prime, so the first direction is clear; conversely, if $P$ has a single
minimal element then every maximal element lies above it and $P=K_{1,t}$, and for $t=0$ the field $\Q$, for
$t\ge1$ the localization of $\Q[x,y]/(f_1)$ given by Theorem~\ref{thm:geometric} is a one-dimensional domain.
For the second statement, an antichain is the spectrum of $\Q^{|P|}$, and otherwise the ring is a localization
of $\Q[x,y]/(f_1\cdots f_n)$ with the $f_i$ distinct irreducibles, so $f_1\cdots f_n$ is squarefree, the
quotient is reduced, and the localization has as its maximal ideals the images of the $\fm_j$; a product with $\Q^{|A_0|}$ is again reduced and
semilocal and does not raise the dimension.
\end{proof}

\begin{proposition}\label{prop:free}
Every finite poset $P$ with no chain of three elements is isomorphic to $\Spec R$ for a ring $R$ that is
essentially of finite type over $\Q$, reduced, semilocal and excellent, and such that for every point $P_j$ of
the construction the localization $R_{\fm_j}$ is a hypersurface quotient of a two-dimensional regular local
ring; in particular $R_{\fm_j}$ is a complete intersection, hence Gorenstein and Cohen--Macaulay.
\end{proposition}

\begin{proof}
If the graph part of $P$ is empty then $P=A_0$ is an antichain, $R=\Q^{|P|}$ is a finite product of fields,
and the assertions are immediate. Otherwise write $F=f_1\cdots f_n$, let $R_0=S^{-1}\bigl(\Q[x,y]/(F)\bigr)$ be
the ring of the construction above and $R=R_0\times\Q^{|A_0|}$. A localization of a finitely generated
$\Q$-algebra is excellent, and so is a quotient of one, so $R_0$
is excellent; and $\Q$ is a quotient of $R_0$, so $R$ is a quotient of
$R_0[u_1,\dots,u_m]/(u_1^2-u_1,\dots,u_m^2-u_m)$ with $m=|A_0|$, since $R_0[u]/(u^2-u)\cong R_0\times R_0$;
hence $R$ is excellent too. The ring $R_0$ is reduced because $F$ is squarefree, and its maximal ideals are the images of the $\fm_j$ by Lemma~\ref{lem:localization}, so $R$ is
semilocal. At $\fm=\fm_j$ the ring $R_\fm$ is $\Q[x,y]_{\fm_j}/(F)$, a
quotient of a two-dimensional regular local ring by a nonzerodivisor; such a ring is a hypersurface, hence a
complete intersection, hence Gorenstein and Cohen--Macaulay. The remaining maximal ideals come from the factors
$\Q$, and a field is a zero-dimensional regular local ring.
\end{proof}

\subsection*{Why localization is needed}

Localization is not an artifact of the construction: a nonzero finitely generated
$k$-algebra with finitely many primes has dimension zero, so the ring of Theorem~\ref{thm:main}(ii) cannot
in general be taken finitely generated. This is \cite[Lemma~10.61.3]{Stacks}; we recall the proof.

\begin{proposition}\label{prop:fg}
Let $k$ be a field and $R$ a nonzero finitely generated $k$-algebra with finitely many primes. Then $\dim R=0$ and
$\Spec R$ is discrete.
\end{proposition}

\begin{proof}
Suppose $\dim R\ge1$ and choose a minimal prime $P$ with $\dim R/P\ge1$, possible because
$\dim R=\sup_P\dim R/P$ over the minimal primes. Then $A=R/P$ is a finitely generated domain of positive
dimension. If $A$ had only finitely many maximal ideals $\fn_1,\dots,\fn_s$, then each $\fn_i\ne0$, so
$\bigcap_i\fn_i\supseteq\prod_i\fn_i\ne0$; choose $0\ne f$ in the intersection. Then $f$ is not a
zero-divisor, so $A_f\ne0$ and $A_f$ has a maximal ideal $\fM$; put $\fp=\fM\cap A$. The composite
$A/\fp\hookrightarrow A_f/\fM$ is injective, and $A_f/\fM$ is a finitely generated $k$-algebra that is a
field, hence finite over $k$ by Zariski's lemma \cite[Proposition~7.9]{AM}. So $A/\fp$ is finite-dimensional
over $k$, hence Artinian, hence a field, and $\fp$ is maximal. But $f/1$ is a unit of $A_f$, so
$f\notin\fp$, whereas every maximal ideal contains $f$. So $A$ has infinitely many maximal ideals, and
therefore $R$ has infinitely many primes: a contradiction. Hence $\dim R=0$; then every prime of $R$ is
maximal, so $\Spec R$ is a finite set of closed points, which is discrete.
\end{proof}

\section{The arithmetic construction}\label{sec:arithmetic}

Here is a second construction, an application of Lemma~\ref{lem:localization} and the arithmetic of
cyclotomic polynomials. It keeps the shape of the ideals of \cite{A} and changes only their exponents, and the arithmetic
behind it explains which neighborhoods those exponents can produce. Nothing later in the paper depends on this
section.

\subsection*{The order of \texorpdfstring{$y$}{y}}

\begin{lemma}\label{lem:order}
Let $k$ be a field, $H\in k[y]$ irreducible with $H(0)\ne0$, and $a\ge1$. Then
$(x^a-y)\subseteq(x-y,H(y))$ if and only if $H$ divides $y^{a-1}-1$. If moreover the class of $y$ has
finite multiplicative order $d$ in the field $k[y]/(H)$, this is equivalent to $d\mid a-1$.
\end{lemma}

\begin{proof}
Let $\pi:k[x,y]\to k[y]$ be the map $x\mapsto y$. Its kernel is $(x-y)$, and for every ideal
$I\subseteq k[y]$ one has $\pi^{-1}(I)\supseteq(x-y)$ and $\pi^{-1}(I)/(x-y)=I$. So $x^a-y$ lies in
$(x-y,H(y))$ exactly when $y^a-y$ lies in $(H)$, that is, exactly when $H$ divides $y(y^{a-1}-1)$. Since $H$
is irreducible, nonconstant and $H(0)\ne0$, it is coprime to $y$, so this is the same as $H\mid y^{a-1}-1$,
that is, $y^{a-1}\equiv1$ in $k[y]/(H)$, which by definition of the order means $d\mid a-1$. Over $\Q$ with
$H=\Phi_d$ the order is exactly $d$, because $\Phi_d$ is the minimal polynomial of a primitive $d$-th root of
unity; then the condition reads $d\mid a-1$.
\end{proof}

\subsection*{The construction}

\begin{theorem}\label{thm:arithmetic}
Let $P=(X\sqcup Y;E)$ be a finite bipartite graph without isolated vertices, with $X=\{x_1,\dots,x_n\}$,
$Y=\{y_1,\dots,y_t\}$ and $N(i)=\{j:(i,j)\in E\}$. Let $d_1,\dots,d_t$ be distinct odd primes, put
\[
a_i=1+2^i\prod_{j\in N(i)}d_j,\qquad p_i=(x^{a_i}-y),\qquad \fm_j=(x-y,\Phi_{d_j}(y)),
\]
and let $R=S^{-1}(\Q[x,y]/(p_1\cdots p_n))$ with $S$ the complement of the union of the images of the
$\fm_j$. Then $\Spec R\cong P$, and $R$ has dimension at most one.
\end{theorem}

\begin{proof}
If $P$ is empty then $R=0$, so suppose $n,t\ge1$. The polynomial $x^{a_i}-y$ is primitive of $y$-degree one over $\Q[x]$, so $p_i$ is prime by the argument
of Lemma~\ref{lem:data}(iii).

Modulo $\fm_j$ the variables $x$ and $y$ are equal, so $\Q[x,y]/\fm_j\cong\Q[y]/(\Phi_{d_j}(y))$ is a
field. Hence $\fm_j$ is maximal. Two of them coincide only if $\Phi_{d_j}\in\fm_{j'}$, which after
substituting $x\mapsto y$ gives $\Phi_{d_{j'}}\mid\Phi_{d_j}$ and hence $d_j=d_{j'}$; so distinct $d_j$ give
distinct $\fm_j$.

By Lemma~\ref{lem:order}, $p_i\subseteq\fm_j$ if and only if $d_j$ divides
$a_i-1=2^i\prod_{j'\in N(i)}d_{j'}$. The $d_j$ are pairwise coprime and odd, so this happens exactly when
$j\in N(i)$.

The numbers $a_i$ are distinct: the $2$-adic valuation of $a_i-1$ is $i$ and its odd part is
$\prod_{j\in N(i)}d_j$. So the $p_i$ are distinct, and they are incomparable: from
$p_i\subseteq p_{i'}$ one gets $x^{a_i}-y=(x^{a_{i'}}-y)w$ with $w\in\Q[x]$, and comparing coefficients of
$y$ forces $w\in\Q^\times$ and $a_i=a_{i'}$.

Every $p_i$ lies in some $\fm_j$ and every $\fm_j$ contains some $p_i$, because all vertices have a
neighbor. Finally, $\Q[x,y]/(x^{a_i}-y)\cong\Q[x]$ is a one-dimensional domain, so the argument of
Lemma~\ref{lem:spectrum} applies: the only primes between $p_i$ and $\fm_j$ are the two ends.

Lemma~\ref{lem:localization} now identifies $\Spec R$ with $P$. The dimension bound follows as before.
\end{proof}

\subsection*{Which neighborhoods the exponents allow}

Fix the moduli $d_1,\dots,d_t$. Lemma~\ref{lem:order} says that the set of $j$ with $p_i\subseteq\fm_j$, which
is the neighborhood of $x_i$, is decided by a single congruence, $d_j\mid a_i-1$: the question is which
families of neighborhoods can arise this way, and the answer is one condition on each of them.

\begin{theorem}\label{thm:moduli}
Let $n,t\ge0$, let $d_1,\dots,d_t\ge2$, and let $N:\{1,\dots,n\}\to\mathcal P(\{1,\dots,t\})$. Then the
following are equivalent.
\begin{enumerate}[label=(\roman*),itemsep=1pt,topsep=2pt]
\item There are integers $a_1,\dots,a_n\ge2$ such that $d_j\mid a_i-1$ exactly when $j\in N(i)$.
\item For every $i$,
\[
N(i)=\bigl\{\,j : d_j \mid \operatorname{lcm}\{d_s : s\in N(i)\}\,\bigr\},
\]
with $\operatorname{lcm}\emptyset=1$.
\end{enumerate}
If (ii) holds, the $a_i$ can moreover be chosen with $a_1<a_2<\cdots<a_n$.
\end{theorem}

\begin{proof}
(i)$\Rightarrow$(ii). Put $L_i=\operatorname{lcm}\{d_s:s\in N(i)\}$, and $L_i=1$ if $N(i)$ is empty. By (i)
each $d_s$ with $s\in N(i)$ divides $a_i-1$, hence $L_i$ divides $a_i-1$. So $d_j\mid L_i$ implies
$d_j\mid a_i-1$, which is $j\in N(i)$; and conversely $j\in N(i)$ gives $d_j\mid L_i$ by definition of $L_i$.

(ii)$\Rightarrow$(i). Let $D=d_1\cdots d_t$. For an index $i$ consider the set
\[
T_i=\{\,1+L_iK \;:\; K\equiv1 \pmod D\,\}.
\]
An element $1+L_iK$ of $T_i$ satisfies $d_j\mid L_iK$ if and only if $d_j\mid L_i$, because $K\equiv1$ modulo
$d_j$ makes $K$ a unit modulo $d_j$; by (ii) this happens exactly for $j\in N(i)$. So every sufficiently large element of $T_i$ is a legitimate choice for $a_i$, and $T_i$ is an unbounded arithmetic progression, hence infinite. Choosing
$a_i\in T_i$ with $a_i>a_{i-1}$ for each $i$ in turn gives $a_1<\cdots<a_n$ and $a_i\ge2$.
\end{proof}

\begin{corollary}\label{cor:coprime}
With the notation of Theorem~\ref{thm:moduli}:
\begin{enumerate}[label=(\roman*),itemsep=1pt,topsep=2pt]
\item if $d_1,\dots,d_t$ are pairwise coprime then (ii) holds for every $N$, so every family of
neighborhoods is available, and the exponents of Theorem~\ref{thm:arithmetic}, with $2$ replaced by a
prime dividing none of the $d_j$, are one choice realizing it;
\item if $d_1\mid d_2\mid\cdots\mid d_t$ are pairwise distinct then (ii) holds exactly when each $N(i)$ is an
initial segment $\{1,\dots,m\}$, so the moduli $d_j=2^{j+1}$ of \cite{A} produce only nested neighborhoods;
\item if (ii) holds and $d_j\mid d_{j'}$ with $j'\in N(i)$, then $j\in N(i)$: divisibility between the moduli
forces monotone neighborhoods.
\end{enumerate}
\end{corollary}

\begin{proof}
(i) If $j\notin N(i)$ then $d_j$ is coprime to every $d_s$ with $s\in N(i)$, hence to their lcm, so
$d_j\nmid L_i$. The exponents $a_i=1+q^i\prod_{j\in N(i)}d_j$ realize $N$ for any prime $q$ coprime to
the $d_j$: $q^i$ is a unit modulo each $d_j$, and the $d_j$ being pairwise coprime,
$d_j\mid\prod_{s\in N(i)}d_s$ exactly when $j\in N(i)$. (ii) In a chain of distinct moduli $\operatorname{lcm}\{d_s:s\in S\}=d_{\max S}$ for
$S\ne\emptyset$ and $1$ for $S=\emptyset$, so $\{j:d_j\mid L_i\}=\{1,\dots,\max N(i)\}$, with
$\max\emptyset=0$. (iii) If $j'\in N(i)$ then $d_{j'}\mid L_i$, hence $d_j\mid L_i$ and
$j\in N(i)$ by Theorem~\ref{thm:moduli}(ii).
\end{proof}

Monotonicity is not sufficient: with $d=(2,3,6)$ the set $\{1,2\}$ is monotone in the sense of (iii), since
neither $d_3\mid d_1$ nor $d_3\mid d_2$, yet it fails (ii), its lcm being $6=d_3$. Coprimality is not
necessary: with $d=(6,10)$ condition (ii) holds for every $N$ although $\gcd(6,10)=2$.

The choice of cyclotomic moduli is not a special one over $\Q$. If $H\in\Q[y]$ is irreducible and the class of
$y$ has finite order in $\Q[y]/(H)$, then $H$ divides $y^d-1$ for some $d$, and $y^d-1$ is squarefree with
irreducible factors the $\Phi_e$, so $H=\Phi_e$ for some $e$; by Lemma~\ref{lem:order} this says that
$(x^a-y)\subseteq(x-y,H(y))$ holds for some $a\ge2$ exactly when $H$ is cyclotomic, and then exactly when
$e\mid a-1$.

Over other fields both constructions still work. The geometric one uses $2t+n$ distinct elements of the field,
that is, after a finite extension if the field is small. The arithmetic one needs an irreducible
$h_j\in k[y]$ whose class has order $d_j$, for which one takes an irreducible factor of $\Phi_{d_j}$ over $k$;
that factor has class of order $d_j$ only when $\operatorname{char}k$ does not divide $d_j$.

The exponents in \cite{A} are $2^i+1$ and $2^j$, so its moduli are $2^{j+1}$. This is not the shape used above, so
the result that applies is Lemma~\ref{lem:order} of this paper, not Theorem~\ref{thm:arithmetic}: the latter
requires the moduli to be distinct odd primes, and the $2^{j+1}$ are not odd. Since $y^{2^j}+1=\Phi_{2^{j+1}}(y)$,
\cite[Lemma~4.1]{A} is the case $a=2^n+1$, $d=2^{m+1}$ of Lemma~\ref{lem:order}, and with these moduli
$d_j\mid a_i-1$ means $2^{j+1}\mid2^i$, that is, $j\le i-1$: the neighborhoods are the initial segments, as in
Corollary~\ref{cor:coprime}(ii). In \cite[Theorem~4.2]{A} the index runs over $i=m+1,\dots,m+n$ with
$m=|Y|$, so each of these segments is all of $Y$ and the poset is $K_{n,m}$: the restriction in the complete
bipartite case lies in the exponents and not in the ideals. The further ideals of \cite[Examples~5.1
and~5.2]{A} have other shapes and give the neighborhoods drawn there.

\section{The converse}\label{sec:converse}

It remains to show that a Noetherian ring with finitely many primes has no chain of three primes.

\begin{lemma}\label{lem:key}
A Noetherian ring of dimension at least two has infinitely many prime ideals of height one.
\end{lemma}

\begin{proof}
We first prove the local case. Let $(T,\fm)$ be a Noetherian local ring of dimension at least two.
If $T$ has infinitely many height-one primes there is nothing to prove, so suppose that $q_1,\dots,q_s$ are
all of them and that $p_1,\dots,p_u$ are its minimal primes. Since $\dim T\ge2$, the maximal ideal is none of
these, so by prime avoidance it contains an element $x$ outside their union, and $x\ne0$ because $0$ lies in
every minimal prime. Let $z$ be a prime minimal over $(x)$. The principal ideal theorem gives $\height z\le1$.
Since $x$ lies in no minimal prime, $z$ is not minimal, hence $\height z=1$, hence $z$ is one of the $q_i$.
But $x\in z$ and $x\notin q_i$.

For the general case, let $R$ be Noetherian of dimension at least two and $p_0\subsetneq p_1\subsetneq p_2$
a chain of primes. The localization $R_{p_2}$ has dimension at least two, and a prime of $R_{p_2}$ and
its contraction have the same height, so the height-one primes of $R_{p_2}$ are primes of $R$ of height
one; the local case gives infinitely many of them.
\end{proof}

\begin{proof}[Proof of Theorem~\ref{thm:main}, (i)$\Rightarrow$(iii)]
Let $R$ be a Noetherian ring with finitely many primes. If $\dim R\ge2$, Lemma~\ref{lem:key} gives infinitely
many primes, so $\dim R\le1$ and there is no chain of three primes. A prime that is neither minimal nor
maximal would be the middle term of such a chain, so every prime is minimal or maximal.
\end{proof}

\section{Beyond finiteness}\label{sec:infinite}

Finiteness enters Theorem~\ref{thm:main} twice, as a hypothesis on $P$ and again in condition (ii), where the
ring is required to be essentially of finite type over a field. The dimension count that rules out a chain of
three elements uses neither, so the question is which \emph{infinite} posets of height at most one are
spectra.

\begin{theorem}\label{thm:infinite}
A poset $P$ with no chain of three elements is isomorphic to $\Spec R$ for a Noetherian ring $R$ if and only
if
\begin{enumerate}[label=(N\arabic*),itemsep=1pt,topsep=2pt]
\item\label{N1} $\min P$ is finite, and
\item\label{N2} any two distinct minimal elements of $P$ have only finitely many common upper bounds.
\end{enumerate}
If these hold and $k$ is an infinite field with $|k|\ge|P|$, then $R$ can be taken reduced, and of dimension
one exactly when $P$ has an edge. For countable $P$ the field $k=\Q$ suffices.
\end{theorem}

The two conditions are independent, and each is needed. The poset $K_{\infty,1}$ satisfies condition~\ref{N2}
and violates condition~\ref{N1}; the poset $K_{2,\infty}$ satisfies condition~\ref{N1} and violates
condition~\ref{N2}. For finite $P$ both hold automatically, so the criterion reduces to
Theorem~\ref{thm:main}(i)$\Leftrightarrow$(iii).

\subsection*{Proof of necessity}

Condition~\ref{N1} is \cite[Lemma~10.31.6]{Stacks}: a Noetherian ring has finitely many minimal primes. For
\ref{N2}, let $p_1\ne p_2$ be minimal primes of $R$. If $p_1$ is maximal, every prime containing $p_1$ equals
$p_1$, and $p_2\not\subseteq p_1$, so the two have no common upper bound. Otherwise $(p_1+p_2)/p_1$ is a
nonzero ideal of the one-dimensional domain $R/p_1$. A prime of $R/(p_1+p_2)$ contains $p_1$, and strictly,
since it also contains $p_2$; it is therefore maximal, a chain of three elements being forbidden. So
$R/(p_1+p_2)$ is zero-dimensional, hence Artinian, hence has finitely many primes, and its primes are the
primes of $R$ that contain both $p_1$ and $p_2$.

\subsection*{Proof of sufficiency}

A poset satisfying \ref{N1} and \ref{N2} is a gluing of stars: apart from the finitely many maximal elements
that lie above more than one minimal element it is a disjoint union of stars, and those points identify
finitely many leaves. Localizations of $k[x]$ realize the stars, and gluing spectra along closed points is a
fiber product of rings, so the two operations match. The rest of this subsection carries out that construction.

Let $A_0=\min P\cap\max P$ be the set of isolated points and $P'=P\smallsetminus A_0$. Both $\min P$ and
$\min P'=\min P\smallsetminus A_0$ are finite by \ref{N1}. If $P'=\emptyset$ then $P$ is a finite antichain
and $k^{|P|}$ realizes it, so suppose $P'\ne\emptyset$ and write $\min P'=\{p_1,\dots,p_s\}$.

For a maximal element $y$ of $P'$ put $M(y)=\{p\in\min P':p\le y\}$. This set is nonempty. First, $y$ is
maximal in $P$: an element $z>y$ would not be minimal, hence would lie in $P'$, contradicting the maximality
of $y$ there. Second, $y$ is not minimal, since otherwise it would be minimal and maximal, hence isolated,
hence not in $P'$. So there is $w<y$, and $w$ is minimal, since an element strictly below $w$ would
give a chain of three elements; as $w$ is not isolated it lies in $P'$, so $w\in M(y)$. By \ref{N2} only
finitely many maximal elements satisfy $|M(y)|\ge2$: each of them is a common upper bound of some pair
$\{p_i,p_j\}$ with $i\ne j$, every pair has only finitely many, and there are finitely many pairs. Write them
as $y_1,\dots,y_T$ and put $M_t=M(y_t)$. Let $n_i$ be the number of maximal elements with $M(y)=\{p_i\}$.

Choose pairwise distinct elements of $k$: an element $a_t$ for each $t\le T$, and, for each $i$ and each of
the $n_i$ maximal elements with $M(y)=\{p_i\}$, one further element. Let $\Lambda_i$ be the set consisting of
the $a_t$ with $i\in M_t$ together with the $n_i$ elements attached to $p_i$. Then
$\Lambda_i\cap\Lambda_j=\{a_t:i,j\in M_t\}$ for $i\ne j$. The number of elements needed is
$T+n_1+\dots+n_s$, the number of maximal elements of $P'$, so $|k|\ge|P|$ allows the choices.

For each $i$ let $B_i$ be the localization of $k[x]$ at the multiplicative set $S_i$ of polynomials that
vanish at no element of $\Lambda_i$. A prime of $k[x]$ survives this localization exactly when it is
contained in $\bigcup_{\lambda\in\Lambda_i}(x-\lambda)$, which happens for the zero ideal and for the ideals
$(x-\lambda)$ with $\lambda\in\Lambda_i$ and for nothing else, an irreducible polynomial with no zero in
$\Lambda_i$ lying in $S_i$. So $B_i$ is a one-dimensional Noetherian domain whose primes are the zero ideal
and these maximal ideals, and $B_i$ contains $k[x]$. Also $\Lambda_i\ne\emptyset$: $p_i$ is not maximal in
$P'$, since otherwise it would be minimal and maximal, hence in $A_0$, so $p_i<y$ for some $y\in P'$, which is
maximal in $P'$ because a chain of three elements is forbidden, and this $y$ contributes either to $n_i$ or
to some $M_t$. In particular $0\notin S_i$, so $B_i\ne0$. Let
\[
R=\Bigl\{(f_1,\dots,f_s)\in\prod_{i=1}^{s}B_i : f_i(a_t)=f_j(a_t)\ \text{for all }t\text{ and all }
i,j\in M_t\Bigr\}.
\]
The conditions make sense because $a_t\in\Lambda_i$ for $i\in M_t$.

The ring $R$ is Noetherian. When $T=0$ there are no conditions and $R=\prod_iB_i$. In general, for a subset
$I\subseteq\{1,\dots,s\}$ put
\[
R(I)=\Bigl\{(f_i)_{i\in I}\in\prod_{i\in I}B_i : f_i(a_t)=f_j(a_t)\ \text{for all }t\text{ and all }
i,j\in M_t\cap I\Bigr\},
\]
so that $R=R(\{1,\dots,s\})$. Each $R(I)$ is Noetherian, by induction on $|I|$: for $|I|=1$ it is one of the
$B_i$, and for $|I|\ge2$ fix $i\in I$ and put $I'=I\smallsetminus\{i\}$. The projections $R(I)\to B_i$ and
$R(I)\to R(I')$ are onto, again by interpolation. The kernel of the first consists of the tuples of $R(I')$
whose $j$th entry vanishes at $a_t$ whenever $i\in M_t$ and $j\in M_t\cap I'$; this is an ideal of
$R(I')$, and its $R(I)$-submodules and its $R(I')$-submodules coincide, $R(I)\to R(I')$ being onto,
so it is a Noetherian $R(I)$-module. Also $B_i$ is a Noetherian $R(I)$-module because $R(I)\to B_i$ is
onto. An extension of two Noetherian modules is Noetherian, so $R(I)$ is Noetherian.

To list the primes, let $\pi_i\colon R\to B_i$ be the $i$th projection and let $\fq_i$ be its kernel. That projection is onto: by Lagrange
interpolation a polynomial takes prescribed values at finitely many points, so a given $f\in B_i$ extends to
a tuple satisfying the conditions. Hence $R/\fq_i\cong B_i$ and $\fq_i$ is prime. Every prime of $R$ contains
one of the $\fq_i$: if a prime $\fp$ contained none, then from $f^{(i)}\in\fq_i\smallsetminus\fp$ we would get
$\prod_if^{(i)}=0\in\fp$ with no factor in $\fp$. Hence every prime of $R$ lies above one of the $\fq_i$, and
the primes above $\fq_i$ correspond to the primes of $B_i$: to the zero ideal, which gives $\fq_i$ back, and
to the maximal ideals $(x-\lambda)$ with $\lambda\in\Lambda_i$, which give
$\fm_{i,\lambda}=\pi_i^{-1}(x-\lambda)$. The $\fq_i$ are minimal, a prime below $\fq_i$ containing some
$\fq_j$ and $\fq_j\subseteq\fq_i$ forcing $j=i$, as the next paragraph shows; the $\fm_{i,\lambda}$ are maximal, because $B_i$ has
dimension one.

It remains to check that this list has no repetitions and no inclusions beyond those forced by the equations
defining $R$. The $\fq_i$ are distinct: if $i\ne j$, the tuple with
$f_i=0$, with $f_j=\prod_{t\,:\,j\in M_t}(x-a_t)\ne0$ and with all remaining entries $0$ satisfies the
conditions and lies in $\fq_i\smallsetminus\fq_j$. The inclusion $\fq_i\subseteq\fm_{j,\lambda}$ holds
exactly when $i=j$, or when $\lambda=a_t$ and $i,j\in M_t$: in those cases
$f_i=0$ forces $f_j(\lambda)=0$, and otherwise one prescribes $f_i=0$ and $f_j(\lambda)\ne0$ and
interpolates the remaining entries. Since $f_i(a_t)=f_j(a_t)$ for every $f\in R$ when
$i,j\in M_t$, the ideals $\fm_{i,a_t}$ agree for all $i\in M_t$. No other two of the
$\fm_{i,\lambda}$ agree: if
$\fm_{i,\lambda}=\fm_{j,\lambda'}$ with $i\ne j$, then $\fq_i\subseteq\fm_{j,\lambda'}$ and
$\fq_j\subseteq\fm_{i,\lambda}$, so $\lambda'=a_t$ and $\lambda=a_u$ with $i,j\in M_t\cap M_u$, and if
$t\ne u$ the tuple with $f_i=x-a_u$, with $f_j$ taking the value $0$ at $a_u$ and $a_t-a_u$ at $a_t$, and
with the remaining entries given by interpolation, lies in $\fm_{i,a_u}$ and not in $\fm_{j,a_t}$.

Counting, $\Lambda_i$ has $|\{t:i\in M_t\}|+n_i$ elements, and the ones indexed by the $a_t$ with $i\in M_t$
are identified with those of the other members of $M_t$. So the maximal elements above $p_i$ alone number
$n_i$, and for each $t$ there is a single maximal element above exactly the members of $M_t$. Hence
$\Spec R\cong P'$.

Finally, $R$ is reduced, being a subring of $\prod_iB_i$, and $\dim R=1$ because a $\lambda\in\Lambda_i$ gives
a chain $\fq_i\subsetneq\fm_{i,\lambda}$. If $A_0\ne\emptyset$, then $R\times k^{|A_0|}$ has the same
properties and spectrum $P'\sqcup A_0=P$.

\begin{remark}\label{rem:infinite-cardinality}
The ring that Theorem~\ref{thm:infinite} builds is a subring of a finite product of localizations of $k[x]$, so for
an infinite field $k$ it has at most $|k|$ elements and hence, being Noetherian, at most $|k|$ primes. From a countable field it
therefore has only countably many. A poset satisfying \ref{N1} and \ref{N2} may nevertheless have uncountably many
maximal elements. Size is not what \ref{N2} measures; what it measures is how many maximal elements can lie above
the same pair of minimal elements.
\end{remark}

\section{How the ring constrains the poset}\label{sec:properties}

Which posets of height at most one remain when the ring is required to be normal, to be a domain, to be a unique
factorization domain, or to be local?

\begin{theorem}\label{thm:normal}
A poset $P$ with no chain of three elements is the spectrum of a Noetherian normal ring if and only if it has
finitely many minimal elements and every maximal element $y$ of $P$ has exactly one minimal element $p$ with
$p\le y$, that is, $P$ is a disjoint union of finitely many stars $K_{1,t}$, with $t$ an
arbitrary cardinal. Noetherian regular rings give the same class.
\end{theorem}

\begin{proof}
Suppose first that $P\cong\Spec R$ with $R$ normal. The ring $R$ is reduced and, being Noetherian, has
finitely many minimal primes
\cite[Lemma~10.31.6]{Stacks}. Since $P\cong\Spec R$ has no chain of three elements, $\dim R\le1$. By
\cite[Lemma~10.37.16]{Stacks}, a reduced ring with finitely many minimal primes is normal if and only if it is
a finite product of normal domains, so $R=\prod_iR_i$ with each $R_i$ a normal domain of dimension at most
one. Such a Noetherian normal domain is a field or a Dedekind domain \cite[Lemma~10.120.17]{Stacks}, so
$\Spec R_i$ is either a point or consists of a unique minimal prime below all its maximal ideals. Since
$\Spec R=\bigsqcup_i\Spec R_i$ with finitely many $i$, $P$ is a disjoint union of finitely many stars, and
every maximal element of $P$ has exactly one minimal element below it.

Conversely, suppose that $P$ has finitely many minimal elements and that every maximal element has exactly one
minimal element below it, so that $P$ is a disjoint union of finitely many copies of $K_{1,t}$. If
$P=\emptyset$, take $R=0$. For a copy with $t\ge1$ take the localization of $k[x]$, for a
field $k$ with $|k|\ge t$, inverting the polynomials that vanish at none of $t$ chosen distinct elements: a
principal ideal domain with exactly $t$ maximal ideals, hence a normal ring whose spectrum is $K_{1,t}$. For
$t=0$ take $\Q$. A finite product of normal rings is normal \cite[Lemma~10.37.15]{Stacks}, and the spectrum of
the product is the disjoint union. Those localizations are regular, and a finite product of regular rings is
regular, so every such poset arises from a Noetherian regular ring; conversely, a Noetherian regular ring of
dimension at most one is normal, its local rings being fields or discrete valuation rings, so no further posets
arise.
\end{proof}

A Noetherian normal ring of dimension at most one is thus a product of finitely many fields and Dedekind
domains, so its spectrum is fixed by the number of factors and, for each, whether it is a field and how many
maximal ideals it has. A domain has one factor.

\begin{theorem}\label{thm:ufd}
For a poset with no chain of three elements to be the spectrum of a Noetherian domain it is necessary and
sufficient that it have a single minimal element, that is, that it be a star $K_{1,t}$ for some cardinal $t$.
The ring can then be taken to be a one-dimensional principal ideal domain, hence a unique factorization domain,
and $\Q$ if $t=0$; Noetherian unique factorization domains give the same spectra.
\end{theorem}

\begin{proof}
A domain has exactly one minimal prime, so a spectrum of this kind has a single minimal element. Conversely,
if $P=K_{1,t}$ with $t\ge1$, the ring that the proof of Theorem~\ref{thm:normal} builds for a single star is a
principal ideal domain with $t$ maximal ideals, hence a unique factorization domain with spectrum $K_{1,t}$;
for $t=0$ take $\Q$. For finite $P$, the first assertion is also
Corollary~\ref{cor:domain}.
\end{proof}

Localizing at a single point instead of forming products leaves one maximal element, and Noetherianness then
keeps the minimal side finite.

\begin{theorem}\label{thm:local}
Among the posets with no chain of three elements, the spectra of Noetherian local rings are exactly the single
point and the finite co-stars $K_{s,1}$; equivalently, exactly the posets with a single maximal element and
finitely many minimal ones. The ring may be taken complete, so the same condition characterizes complete
Noetherian local rings.
\end{theorem}

\begin{proof}
A local ring has a unique maximal ideal and, being Noetherian, finitely many minimal primes, so a spectrum of
this kind is a single point or $K_{s,1}$ with $s$ finite. The ring constructed below is complete, so it
suffices to realize those posets. If $P$ is a single point take a field. Otherwise $P=K_{s,1}$ with $s\ge1$ finite, and
we take $R=k[[x,y]]/(f_1\cdots f_s)$ with $f_i=y-g_i(x)$, where $g_1,\dots,g_s\in k[[x]]$ are distinct with
$g_i(0)=0$. As a quotient of a complete regular local ring, $R$ is a complete local ring. The ring $k[[x,y]]$
is a unique factorization domain, so $\prod_i f_i\in\fq$ with $\fq$ prime gives $f_i\in\fq$ for some $i$, hence
$(f_i)\subseteq\fq$; and $k[[x,y]]/(f_i)\cong k[[x]]$ is a discrete valuation ring, whose primes are $0$ and
$(x)$, the latter pulling back to $(x,y)$ because $g_i(0)=0$. The ideals $(f_i)$ are distinct: if
$(f_i)=(f_j)$ with $i\ne j$, then $f_j\in(f_i)$, but the isomorphism $k[[x,y]]/(f_i)\cong k[[x]]$ sends $f_j$ to
$g_i-g_j\ne0$. So $\Spec R=\{(f_i)\}\cup\{(x,y)\}$, which is $K_{s,1}$.
\end{proof}

So normal rings give finite disjoint unions of stars, domains a single star, and local rings a point or a finite
co-star. Beyond height one the picture is incomplete: no finite poset of height two is a Noetherian spectrum by
Theorem~\ref{thm:main}, and which infinite posets of height two are Noetherian spectra is open
\cite[\S1]{BL}. Section~\ref{sec:heighttwo} takes up the infinite case.

\section{Height two}\label{sec:heighttwo}

In height two, one obstruction is immediate. By Lemma~\ref{lem:key} a Noetherian ring of
dimension at least two has infinitely many height-one primes. Applied to $R_q/pR_q$ this says
that if $p\subsetneq q$ are primes with $\operatorname{ht}(q/p)\ge2$, then infinitely many primes $r$
satisfy $p\subsetneq r\subsetneq q$ with $\operatorname{ht}(r/p)=1$. So no interval of a
Noetherian spectrum of height at least two is finite, and every element of height at least two
has infinitely many elements below it. This much is classical. What it leaves open, and what the
rest of the section measures, is how the height-one elements are distributed over the maximal
ones.

For a poset $P$ of height exactly two with a least element $\bot$ and
$\max P=\{m_1,\dots,m_t\}$ finite and $\operatorname{ht}m_j=2$ for every $j$, and for
$\emptyset\ne S\subseteq\{1,\dots,t\}$, let $A_S$ be the set of height-one elements $x$ with
$x<m_j$ exactly for $j\in S$. The sets $A_S$ are disjoint and cover the height-one elements. The hypothesis
$\operatorname{ht}m_j=2$ is not redundant: a maximal element of height one has
$A_{\{j\}}=\emptyset$, so the conclusion of Proposition~\ref{thm:heighttwo}(ii) would fail.

The section has three steps: the criterion in the countable case
(Proposition~\ref{thm:heighttwo} and Corollary~\ref{cor:countableiff}), the cardinality
obstruction (Proposition~\ref{thm:barrier} and Corollary~\ref{cor:barrier}), and the two
consequences that measure the distance from the finite case (Corollaries~\ref{cor:cones}
and~\ref{cor:nonlocal}). Remark~\ref{rem:ww} records that the witness of
Corollary~\ref{cor:barrier} refutes Conjecture~1.4 of \cite{WW}. The two questions left
open close the section. It can be read independently of Sections~\ref{sec:infinite}
and~\ref{sec:properties}.

\begin{proposition}\label{thm:heighttwo}
Let $P$ be a poset of height exactly two with a least element $\bot$, with
$\max P=\{m_1,\dots,m_t\}$ finite and $t\ge1$, and with $\operatorname{ht}m_j=2$ for every
$j$. For $\emptyset\ne S\subseteq\{1,\dots,t\}$ let $A_S$ be the set of height-one elements
$x\in P$ with $x<m_j$ exactly for $j\in S$.
\begin{enumerate}
\item[(i)] If $P$ is countable and $A_S$ is infinite for every nonempty $S$, then
$P\cong\Spec R$ for a ring $R$ essentially of finite type over $\Q$.
\item[(ii)] If $P\cong\Spec R$ for a Noetherian ring $R$, then $A_{\{j\}}$ is infinite for
every $j=1,\dots,t$.
\end{enumerate}
\end{proposition}

\begin{proof}
For (ii), since $P$ has a least element, $R$ has a unique minimal prime $p$ and
$\Spec R=\Spec(R/p)$ as posets, heights being preserved. So $R$ may be taken to be a domain of
dimension two with maximal ideals $m_1,\dots,m_t$, each of height two. Fix $j$ and suppose that
the height-one primes lying inside $m_j$ and in no other maximal ideal are $q_1,\dots,q_s$. Apply
prime avoidance to $(0)$, the $m_{j'}$ with $j'\ne j$ and the $q_i$. None of these
contains $m_j$, so there is $f\in m_j$ outside all of them, and $f\ne0$. A prime $q$ minimal
over $(f)$ has height at most one by Krull's principal ideal theorem, and height not zero
because $f\ne0$ and $R$ is a domain, so height one. It lies in the only maximal ideal containing
$f$, namely $m_j$. So $q\in A_{\{j\}}$, and $q\ne q_i$ because $f\notin q_i$, contradicting
the enumeration $q_1,\dots,q_s$. Hence $A_{\{j\}}$ is infinite.

For (i), choose $P_j=(a_j,b_j)\in\Q^2$ with distinct $a_j$ and let $R$ be the localization of
$\Q[x,y]$ at the polynomials vanishing at none of the $P_j$. By Lemma~\ref{lem:localization}
the primes of $R$ are $(0)$, the ideals of the $P_j$, and the height-one primes through at
least one $P_j$. These last are principal, $\Q[x,y]$ being a unique factorization domain.
Interpolate $g_0$ with $g_0(a_j)=b_j$ and put, for nonempty $S$ and $c\in\Q^\times$,
\[
f_{c,S}=y-g_0(x)-c\prod_{j\in S}(x-a_j).
\]
Being monic of degree one in $y$, $f_{c,S}$ is irreducible, and $f_{c,S}(P_j)$ is zero exactly
for $j\in S$. So $(f_{c,S})$ is a height-one prime of $R$ lying below $m_j$ exactly for
$j\in S$. If $(f_{c,S})=(f_{c',S'})$, then $S=S'$, both curves passing through the same points
$P_j$. If $S=S'$ and $c\ne c'$, then $f_{c,S}$ and $f_{c',S}$ are distinct irreducibles, hence
coprime in the unique factorization domain $\Q[x,y]$. An equality $f_{c,S}=u f_{c',S}$ with
$u=a/s$ a unit of $R$ would give $s f_{c,S}=a f_{c',S}$, so that $f_{c',S}$ divides $s$. But
$f_{c',S}$ vanishes at the $P_j$ with $j\in S$, which $s$ does not. Hence $c=c'$. Each class of
$\Spec R$ contains the countably many distinct primes $(f_{c,S})$, and $\Spec R$ is countable
because $R$, a localization of the countable ring $\Q[x,y]$, is countable and Noetherian, so each of
its ideals is generated by a finite subset of $R$, and there are only countably many finite subsets. Since $P$ is countable and every $A_S$
is infinite, every class of $P$ is countably infinite, so $\Spec R\cong P$. Finally $R$ is a
localization of $\Q[x,y]$, hence essentially of finite type over $\Q$.
\end{proof}

Colbert \cite[Definition~4.1]{C18} calls a poset $U$ of height at most two a \emph{K-poset} when
$\min U$ and the set of its height-two elements are finite, when $\operatorname{mub}(u,v)$ is finite for
all distinct $u,v\in\min U$, and when for every chain $u>v>w$ of three elements the class $[u/w]$, the
set of elements whose only element above them is $u$ and whose only element below them is $w$, is
infinite. A K-poset is \emph{proper} when for every maximal $u$ and minimal $w$ the class $[u/w]$ is
empty or has cardinality $|U|$. Here the only
chains of three elements are $m_j>x>\bot$, and $[m_j/\bot]=A_{\{j\}}$, since a height-one
element has nothing above it but maximal elements and nothing below it but $\bot$. So the third
condition says exactly that the single classes are infinite, which is the conclusion of (ii).
\cite[Theorem~4.2]{C18} records that it holds in spectra of Noetherian two-dimensional rings with
finitely many height-two primes.

\begin{corollary}\label{cor:countableiff}
Let $P$ be countable, of height exactly two, with a least element $\bot$ and with
$\max P=\{m_1,\dots,m_t\}$ finite and $\operatorname{ht}m_j=2$ for every $j$. Then
$P\cong\Spec R$ for some Noetherian ring $R$ if and only if $A_{\{j\}}$ is infinite for every
$j$.
\end{corollary}

\begin{proof}
Necessity is Proposition~\ref{thm:heighttwo}(ii). For sufficiency, $P$ is a proper K-poset:
it has height two, the sets of minimal and of height-two elements are finite, the latter because every
element of height two is maximal, an element of height two below a maximal element giving a
maximal element of height at least three. The condition on minimal upper bounds is vacuous,
$\bot$ being the only minimal element. The only pairs $(u,w)$ with $u$ maximal, $w$ minimal and
$u>w$ are the $(m_j,\bot)$, and $[m_j/\bot]=A_{\{j\}}$ is an infinite subset of the countable
set $P$, hence of cardinality $|P|$. The only chains of three elements are $m_j>x>\bot$, so the
third condition of a K-poset holds as well. Hence \cite[Theorem~1.1]{C18} applies.
\end{proof}

The criterion is the countable case, for posets whose maximal elements all have height two, of
Wiegand's classification \cite[Theorem~2]{W83}, as \cite[p.~1]{C18} records. For countable $P$ of this kind
the answer is complete. The hypothesis of
Proposition~\ref{thm:heighttwo}(i) that every class be infinite is stronger than necessary: by
Corollary~\ref{cor:countableiff} the single classes suffice, and Remark~\ref{rem:intermediate}
shows that the remaining classes may even be empty. It is retained in (i) only because the
construction there gives rings essentially of finite type over $\Q$, which the appeal to
\cite[Theorem~1.1]{C18} does not.

For $t=1$ the two halves of Proposition~\ref{thm:heighttwo} meet, and there the answer is
complete at every cardinality (Corollary~\ref{cor:onemaximal}).

\begin{corollary}\label{cor:onemaximal}
Let $P$ be a poset of height exactly two with a least element and a single maximal element. Then
$P$ is isomorphic to $\Spec R$ for some Noetherian ring $R$ if and only if the set of
height-one elements of $P$ is infinite.
\end{corollary}

\begin{proof}
If $A_{\{1\}}$ is infinite then $|A_{\{1\}}|=|P|$ and $P$ is a proper K-poset, as in the proof of
Corollary~\ref{cor:countableiff}, so \cite[Theorem~1.1]{C18} realizes $P$, whatever the cardinality
of $P$. If $A_{\{1\}}$ is finite, Proposition~\ref{thm:heighttwo}(ii) forbids $P$.
\end{proof}

\begin{proposition}[cardinality barrier]\label{thm:barrier}
Let $P$ be a poset of height exactly two with a least element $\bot$ and with
$\max P=\{m_1,\dots,m_t\}$ finite and $\operatorname{ht}m_j=2$ for every $j$, and let $A_S$ be as
in Proposition~\ref{thm:heighttwo}. If $P\cong\Spec R$ for a Noetherian ring $R$, then for every
nonempty $S\subseteq\{1,\dots,t\}$ and every $j\in\{1,\dots,t\}$,
\[
|A_S|\;\le\;|\Spec R|\;\le\;|R|\;\le\;|R/m_j|^{\aleph_0}\;\le\;|A_{\{j\}}|^{\aleph_0}.
\]
\end{proposition}

\begin{proof}
Since $P$ has a least element, $R$ has a unique minimal prime $p$ and $\Spec(R/p)=\Spec R$ as posets,
heights being preserved; so $\bar R=R/p$ is a domain of dimension two whose maximal ideals we again
call $m_1,\dots,m_t$, each of height two, and $R/m_j\cong\bar R/m_j$. The prime $p$ is nilpotent, and
the successive quotients of the filtration $0=p^n\subseteq\cdots\subseteq p\subseteq R$ are finitely
generated $\bar R$-modules, so $|R|\le|\bar R|^{\aleph_0}$. It therefore suffices to prove the chain for
the domain $\bar R$, in which the classes $A_S$ are the same sets.

\emph{First inequality.} $A_S$ consists of primes of $\bar R$.

\emph{Second inequality.} Every ideal of the Noetherian ring $\bar R$ is generated by a finite subset,
so the ideals are the images of the finite subsets of $\bar R$; as $\bar R$ is infinite here, there are
at most $|\bar R|$ of them.

\emph{Third inequality.} By Krull's intersection theorem $\bigcap_{n\ge1}m_j^n=(0)$, so $\bar R$ embeds
in the completion $\hat R=\varprojlim_n\bar R/m_j^n$. Each $\bar R/m_j^n$ is a local Artinian ring whose
only prime is $m_j/m_j^n$, so all its composition factors are $k_j:=\bar R/m_j$ and its cardinality is
$|k_j|^{\ell_n}$, with $\ell_n$ the length. Hence
\[
|\bar R|\;\le\;|\hat R|\;\le\;\prod_{n\ge1}|\bar R/m_j^n|\;=\;|k_j|^{\sum_n\ell_n}\;\le\;|k_j|^{\aleph_0},
\]
the last step because every $\ell_n$ is finite and a countable sum of finite numbers is at most $\aleph_0$.

\emph{Fourth inequality.} It suffices to inject $k_j$ into $A_{\{j\}}$. Choose
$0\ne g\in\bigcap_{i=1}^tm_i$, possible because $\bar R$ is a domain and $t$ is finite, and let
$Q_1,\dots,Q_s$ be the minimal primes of $(g)$, each of height one by Krull's principal ideal theorem.
Since $M:=m_j$ is maximal and distinct from the $m_i$ with $i\ne j$, and
$\operatorname{ht}Q_l=1<2=\operatorname{ht}M$, prime avoidance gives
$f\in M\setminus\bigl(\bigcup_{i\ne j}m_i\cup\bigcup_{l=1}^sQ_l\bigr)$. Fix a set $Z\subseteq\bar R$ of
representatives of $k_j$ and put $w_z=f+zg$ for $z\in Z$. Then $w_z\in M$ and $w_z\ne0$, since
$f=-zg\in(g)\subseteq Q_1$ would otherwise contradict the choice of $f$. Let $q_z$ be a minimal prime of
$(w_z)$; one lies in $M$, because $(w_z)\subseteq M$ and $M$ is prime, and it has height one:
$\operatorname{ht}q_z\le1$ by Krull's principal ideal theorem, and $\operatorname{ht}q_z\ne0$ because
$w_z\ne0$ and $\bar R$ is a domain. Moreover $q_z$ lies in no maximal ideal other than $M$: if
$q_z\subseteq m_i$ with $i\ne j$, then $w_z\in m_i$, whereas $f\notin m_i$ and $zg\in m_i$. So
$q_z\in A_{\{j\}}$, and $z\mapsto q_z$ is injective: if $q_z=q_{z'}=q$ with $z\ne z'$, then
$(z-z')g\in q$, while $z-z'\notin M\supseteq q$, so $g\in q$ because $q$ is prime, whence $q$ is a
height-one prime containing $g$, that is $q=Q_l$ for some $l$, and then $f=w_z-zg\in Q_l$, again contrary
to the choice of $f$. Hence $|A_{\{j\}}|\ge|Z|=|k_j|$. Since
$|\bar R|\le|R|\le|\bar R|^{\aleph_0}\le(|k_j|^{\aleph_0})^{\aleph_0}=|k_j|^{\aleph_0}$ and the classes
are the same sets, the chain holds for $R$.
\end{proof}

The chain is \cite[Lemma~8.2]{C18}, stated there for a single maximal ideal $M$ of height two in a
semilocal Noetherian domain of dimension two, with the last step of that lemma attributed to \cite[Lemma~2.1]{KO};
the proof above reproduces it in the form needed here, $S$ being arbitrary and $j$ not required to lie
in $S$. What is new, as far as we know, is Corollary~\ref{cor:barrier} and the use made of the chain,
not the chain.

\begin{corollary}\label{cor:barrier}
There is a poset $Q$ of height exactly two, with a least element, with two maximal elements of
height two and with every class $A_S$ infinite, which is not isomorphic to $\Spec R$ for any
Noetherian ring $R$.
\end{corollary}

\begin{proof}
Let $Q$ be the poset with a least element $\bot$, two maximal elements $m_1,m_2$ of height two,
classes $A_{\{1\}}$ and $A_{\{1,2\}}$ countably infinite and $A_{\{2\}}$ of cardinality
$(2^{\aleph_0})^+$, with $\bot$ below every element and $m_j$ above the elements of $A_S$
exactly when $j\in S$. Then every $A_S$ is infinite and $|Q|=(2^{\aleph_0})^+$. If
$Q\cong\Spec R$, Proposition~\ref{thm:barrier} with $S=\{2\}$ and $j=1$ gives
$(2^{\aleph_0})^+=|A_{\{2\}}|\le|A_{\{1\}}|^{\aleph_0}=2^{\aleph_0}$, contrary to
$(2^{\aleph_0})^+>2^{\aleph_0}$.
\end{proof}

The witness $Q$ is a K-poset, its only chains of three elements being $m_j>x>\bot$, and for
these $[m_j/\bot]=A_{\{j\}}$ is infinite. But $Q$ is not proper, $[m_1/\bot]=A_{\{1\}}$ having
$\aleph_0$ elements while $|Q|=(2^{\aleph_0})^+$. It also satisfies $|A_S|\le|A_{\{j\}}|$
whenever $j\in S$, so no condition of that shape would rule it out. Hence the properness
hypothesis of \cite[Theorem~1.1]{C18} cannot be deleted, and neither can the countability
hypothesis of Proposition~\ref{thm:heighttwo}(i).

\begin{remark}\label{rem:ww}
Conjecture~1.4 of \cite{WW} asserts that a two-dimensional poset with finite $\min$ and finite $\max$,
dimension here being the supremum of the heights of the elements, is isomorphic to the spectrum of a
Noetherian ring if and only if $L_e(u)$ is infinite for every $u$ of height two and
$\operatorname{mub}(u,v)$ is finite for all $u,v\in\min$, where $L_e(u)$ is the set of elements whose
only strict upper bound is $u$ and $\operatorname{mub}$ the set of minimal upper bounds. The witness
$Q$ of Corollary~\ref{cor:barrier} satisfies both conditions: $L_e(m_j)=A_{\{j\}}$ is infinite for
$j=1,2$, and $\operatorname{mub}(\bot,\bot)=\min(Q)=\{\bot\}$ is finite, $\bot$ being the only minimal
element. Yet $Q$ is not the spectrum of a Noetherian ring, so the conjecture fails at $Q$. What excludes
$Q$ is the chain of Proposition~\ref{thm:barrier}, which neither condition sees; when $P$ is countable, has a least element and finitely many
maximal elements, all of height two, the conjecture reduces to Corollary~\ref{cor:countableiff}.
\end{remark}

\begin{remark}\label{rem:intermediate}
The classes $A_S$ with $|S|\ge2$ need not be infinite. Let $B$ be the poset with a least
element, two maximal elements $m_1,m_2$ of height two, and two countably infinite classes
$A_{\{1\}}$ and $A_{\{2\}}$, so that $A_{\{1,2\}}=\emptyset$. Call it the \emph{bowtie}. The
bowtie is a proper K-poset, since the sets of minimal and of height-two elements are finite, the
condition on minimal upper
bounds is vacuous, and $[m_j/\bot]=A_{\{j\}}$ has cardinality $|B|$. So
\cite[Theorem~1.1]{C18} realizes it. Earlier examples of two
height-two maximal ideals with no common height-one prime are due to Heitmann
\cite[Example~2]{He77} and McAdam \cite{Mc76}.
\end{remark}

\begin{corollary}\label{cor:cones}
Let $t\ge1$ and let $A$ be an infinite antichain whose elements all lie below every $m_j$, with
$\bot$ below everything. Then $P=\{\bot\}\cup A\cup\{m_1,\dots,m_t\}$ is a Noetherian spectrum
if and only if $t=1$.
\end{corollary}

\begin{proof}
For $t\ge2$ no height-one element lies below exactly one maximal element, so
$A_{\{j\}}=\emptyset$ for every $j$ and Proposition~\ref{thm:heighttwo}(ii) forbids $P$. For
$t=1$ the set $A_{\{1\}}=A$ is the only class, and $|A|=|P|$. So $P$ is a proper K-poset, as
in the proof of Corollary~\ref{cor:countableiff}, and \cite[Theorem~1.1]{C18} realizes it.
\end{proof}

\begin{corollary}\label{cor:nonlocal}
There are posets $P_1,P_2$ of height two with the same multiset of interval isomorphism types,
such that $P_2$ is a Noetherian spectrum, $P_1$ is not, and every interval of $P_1$ is a
Noetherian spectrum.
\end{corollary}

\begin{proof}
Take for $P_1$ the poset of Corollary~\ref{cor:cones} with $t=2$ and $A$ countably infinite,
and for $P_2$ the poset obtained by splitting its middle layer into three countably infinite
classes. Every interval of either poset is a point, a two-element chain, or a double cone
$\{\bot\}\cup D\cup\{m_j\}$ with $D$ countably infinite, and the double cones are spectra by
Corollary~\ref{cor:cones}. The two multisets of interval types are both
$(\aleph_0,\aleph_0,2)$, counting points, two-element chains and double cones. So $P_2$ is a
spectrum by Corollary~\ref{cor:countableiff}, while $P_1$ is not, by
Corollary~\ref{cor:cones}.
\end{proof}

So realizability in height two is not a function of the multiset of interval isomorphism types:
no criterion phrased in terms of that multiset alone can be correct, in contrast with the finite
case of Theorem~\ref{thm:main}.

Two questions are left open. First, whether the last inequality of Proposition~\ref{thm:barrier}
is sufficient: is a poset of height two with a least element, finitely many maximal elements, all
of height two, and with every $A_{\{j\}}$ infinite a spectrum whenever
$|A_S|\le|A_{\{j\}}|^{\aleph_0}$ for every nonempty $S\subseteq\{1,\dots,t\}$ and every $j\in\{1,\dots,t\}$? By
Corollary~\ref{cor:countableiff} a counterexample has at least $\aleph_1$ elements, and the
bowtie of Remark~\ref{rem:intermediate}, with $A_{\{2\}}$ enlarged to $\aleph_1$, has exactly
$\aleph_1$ elements and satisfies $|A_S|\le|A_{\{j\}}|^{\aleph_0}$ for every nonempty $S\subseteq\{1,\dots,t\}$ and every $j\in\{1,\dots,t\}$. Second, whether the exponent in Proposition~\ref{thm:barrier} is needed: if this poset
is a spectrum, then
$|A_S|\le|A_{\{j\}}|$ fails for $S=\{2\}$ and $j=1\notin S$, so the exponent cannot be
removed. If it is not a spectrum, the inequalities are not sufficient. Either outcome settles
one of the two questions. We do not know, however, whether the enlarged bowtie is a Noetherian
spectrum. It is not proper: the class $[m_1/\bot]=A_{\{1\}}$ has $\aleph_0$ elements while the
whole poset has $\aleph_1$, so \cite[Theorem~1.1]{C18} does not apply.

Section~\ref{sec:relation} relates these results to the literature.

\section{Relation to earlier work}\label{sec:relation}

The results of \cite{A} are the instances of Theorem~\ref{thm:main} in which the graph is complete bipartite,
a path, an even cycle, a disjoint union of these, or one of the two families added in Examples~5.1
and~5.2 of that paper: for $K_{m,n}$ every neighborhood is full, for a path the neighborhoods are
$\{1\},\{1,2\},\{2,3\},\dots$, and for an even cycle they are $\{1,n\}$ and $\{i-1,i\}$.

The distinction between a graph and a poset matters here. The introduction of \cite{A} says that the Hasse
diagram of every finite Noetherian spectrum is bipartite. That is true and it is not a characterization: the
diamond $\{a<b,\ a<c,\ b<d,\ c<d\}$ has the same bipartite Hasse diagram $C_4$ as $K_{2,2}$, and only the
second is a Noetherian spectrum. What excludes the diamond, and the three-element chain, is a chain of three
elements; what Theorem~\ref{thm:main} supplies is that this one condition is also sufficient. The closing
paragraph of \cite{A} anticipates the characterization, provided its graph is read in this sense.

The implication (i)$\Rightarrow$(iii) is classical: it is \cite[Lemma~10.61.2]{Stacks}, quoted as a well-known
property of Noetherian rings in \cite[\S1]{BL}. Theorem~\ref{thm:infinite} is in substance Colbert's: its two
conditions are, for posets of height at most one, the first two in his definition of a K-poset
\cite[Definition~4.1]{C18}; see also \cite{WW}. A K-poset of height at
most one is automatically proper, \cite[Theorem~1.1]{C18} realizes every proper K-poset, and the conditions hold
in every Noetherian spectrum of dimension at most one, as Section~\ref{sec:infinite} shows. The explicit
construction of (iii)$\Rightarrow$(ii), and with it the essentially of finite type form of clause (ii) of
Theorem~\ref{thm:main} that settles Conjecture~5.3 of \cite{A}, is what this paper contributes: Colbert's proof
is written for posets of dimension two --- its necessity reduces to rings whose maximal ideals have height two
--- and its construction, which pastes localizations of polynomial rings in possibly infinitely many
indeterminates, does not give essentially of finite type rings in general.

Section~\ref{sec:heighttwo} reaches less far. The chain of Proposition~\ref{thm:barrier} is
\cite[Lemma~8.2]{C18}, whose proof Section~\ref{sec:heighttwo} reproduces in the form needed here: the
trivial inequality $|A_S|\le|\Spec R|$ is put in front, the $\aleph_0$ term of that lemma is dropped, and
its first inequality, raised to the $\aleph_0$, is used for the last step. In the
countable case with a single minimal element and finitely many maximal elements, all of height two, the
classification is Wiegand's \cite[Theorem~2]{W83}. What the section adds is
the elementary construction of Proposition~\ref{thm:heighttwo}(i), the counterexample of
Corollary~\ref{cor:barrier} --- a K-poset with every class infinite that would violate the chain if it were a
spectrum, so that neither the countability hypothesis of Proposition~\ref{thm:heighttwo}(i) nor the properness
hypothesis of \cite[Theorem~1.1]{C18} can be dropped --- and the observation of
Corollary~\ref{cor:nonlocal} that the multiset of interval isomorphism types does not determine
realizability. We have not found this observation in the literature.

The closest results are embedding theorems. Every finite poset embeds in the spectrum of a Noetherian ring so
that saturated chains are preserved \cite{He79}, and in the spectrum of a Noetherian unique factorization
domain \cite{CL}, and \cite{BL} characterize when such an embedding can preserve the dimension as well. In these
theorems the spectrum is in general strictly larger than the poset; for the unique factorization domain
versions, by Theorem~\ref{thm:ufd} the two can coincide only when the poset has a single minimal element, so
they cannot be sharpened to equalities in general.
Nor have we found in the literature the poset forms of Theorems~\ref{thm:normal}
and~\ref{thm:local}, or the criterion of Theorem~\ref{thm:moduli} deciding which families of
neighborhoods a given choice of moduli allows. The survey \cite{WW} describes the surrounding questions; its Conjecture~1.4 is refuted by the witness
of Corollary~\ref{cor:barrier} (Remark~\ref{rem:ww}), a counterexample we have not found recorded. Heitmann's earlier note \cite{He77} pastes Noetherian $G$-rings together so that the nonzero primes
of the result are the disjoint union of the nonzero primes of the pieces. Its Example~2 is an earlier
example of two height-two maximal ideals with no common height-one prime, cited in
Section~\ref{sec:heighttwo}.

Without the Noetherian hypothesis nothing of this survives: by \cite{H} the spectral spaces are exactly the
spectra of rings, so every finite poset with the Alexandrov topology is a spectrum, and \cite{Stacks} gives a finite spectrum of dimension greater than one. The Noetherian hypothesis is what
forces the constraint of Theorem~\ref{thm:main}.

\section*{Declaration of generative AI and AI-assisted technologies in the manuscript preparation process}

During the preparation of this work the author used DeepSeek V4.1 Flash to identify key literature, to
improve the readability and organization of the manuscript, and to check some details of proofs and
calculations. After using this tool, the author reviewed and edited the content as needed and takes full
responsibility for the verification and the exposition of this article.

\end{document}